\documentclass{amsart}
\usepackage{amsfonts}
\usepackage{amsmath,amssymb}
\usepackage{amsthm}
\usepackage{amscd}
\usepackage{graphics}
\usepackage{graphicx}
\theoremstyle{remark}{
\newtheorem{Def}{{\rm Definition}}
\newtheorem{Ex}{{\rm Example}}
\newtheorem{Rem}{{\rm Remark}}

}
\newtheorem{Cor}{Corollary}
\newtheorem{Prop}{Proposition}
\newtheorem{Thm}{Theorem}

\begin{document}
\title[Increasing two connected components of singular sets of fold maps]{Surgery operations to fold maps to increase connected components of singular sets by two}
\author{Naoki Kitazawa}
\keywords{Singularities of differentiable maps; generic maps. Differential topology. Reeb spaces. {\it \textup{2020} Mathematics Subject Classification}: Primary~57R45. Secondary~57R19.}
\address{Institute of Mathematics for Industry, Kyushu University, 744 Motooka, Nishi-ku Fukuoka 819-0395, Japan\\
 TEL (Office): +81-92-802-4402 \\
 FAX (Office): +81-92-802-4405}
\email{n-kitazawa@imi.kyushu-u.ac.jp}
\urladdr{https://naokikitazawa.github.io/NaokiKitazawa.html}
\maketitle
\begin{abstract}
In geometry, understanding the topologies and the differentiable structures of manifolds in constructive ways is fundamental and important. It is in general difficult, especially for higher dimensional manifolds.

The author is interested in this and trying to understand manifolds via construction of explicit {\it fold} maps: differentiable maps locally represented as product maps of Morse functions and identity maps on open balls. Fold maps have been fundamental and useful in investigating the manifolds by observing (the sets of) singular points and values and preimages as Thom and Whitney's pioneering studies and recent studies of Kobayashi, Saeki, Sakuma, and so on, show. Here, construction of explicit fold maps on explicit manifolds is difficult.
 
The author constructed several explicit families of fold maps and investigated the manifolds admitting the maps. Main fundamental methods are surgery operations ({\it bubbling operations}), the author recently introduced motivated by Kobayashi and Saeki's studies such as operations to deform generic differentiable maps whose codimensions are negative into the plane preserving the differentiable structure of the manifold in 1996 and so on. We remove a neighborhood of a (an immersed) submanifold consisting of regular values in the target space, attach a new map and obtain a new fold map such that the number of connected components of the set consisting of singular points increases. In this paper, we investigate cases where the numbers increase by two and obtain cases of a new type. 
          

\end{abstract}


\maketitle
\section{Introduction and fundamental notation and terminologies.}
\label{sec:1}
In geometry, understanding the topologies and the differentiable structures of manifolds in geometric and constructive ways is fundamental and important. It is in general difficult, especially for higher dimensional manifolds, although they were considerably understood in the latter half of the last century via sophisticated algebraic topological theory and abstract differential topological theory such as homotopy theory, fundamental theory of Morse functions, surgery theory, and so on. See also \cite{milnor}, \cite{milnor2} and \cite{wall} for example to know related tools more precisely.

This paper concerns this understanding via {\it fold} maps, regarded as higher dimensional versions of Morse functions.

We first review fold maps. Before this, we explain several fundamental terminologies on differentiable manifolds and maps.

Throughout this paper, for a differentiable map, a {\it singular point} is a point at which the rank of the differential of the map drops. a {\it singular value} is a point in the target space realized as a value at a singular point, and a {\it regular value} is a point in the target space which is not a singular value: they are defined as this in most of introductory books on differentiable manifolds and maps.

Moreover, the {\it singular set} of the map is the set of all singular points, the {\it singular value set} is the image of the singular set, and the {\it regular value set} is the complement of the singular value set.

Throughout the present paper, manifolds are differentiable and smooth (of class $C^{\infty}$), maps are differentiable and smooth (of class $C^{\infty}$) unless otherwise stated. A diffeomorphism is always assumed to be smooth and the {\it diffeomorphism group} of a manifold means the group consisting of all diffeomorphisms on the manifold.
 
For a smooth manifold $X$, we denote the tangent bundle by $TX$ and the tangent vector at $p \in X$ by $T_p X \subset TX$.

\subsection{Fold maps.}
{\it Fold} maps have been fundamental and useful in investigating the manifolds by observing the singular sets and the singular value sets and preimages as Thom and Whitney's pioneering studies (\cite{thom} and \cite{whitney}) and recent studies of Kobayashi, Saeki, Sakuma, and so on, which we will introduce later, show.

\begin{Def}
\label{def:1}
Let $m>n \geq 1$ be integers. A smooth map from an $m$-dimensional smooth manifold with no boundary into an
$n$-dimensional smooth manifold with no boundary is said to be a {\it fold} map if at each singular point $p$, the map is represented as
$$(x_1, \cdots, x_m) \mapsto (x_1,\cdots,x_{n-1},\sum_{k=n}^{m-i}{x_k}^2-\sum_{k=m-i+1}^{m}{x_k}^2)$$
for some coordinates and an integer $0 \leq i(p) \leq \frac{m-n+1}{2}$.
\end{Def}
\begin{itemize}
\item For any singular point $p$, the $i(p)$ in Definition \ref{def:1} is unique : $i(p)$ is the {\it index} of $p$.
\item The set consisting of all singular points of a fixed index of the map in Definition \ref{def:1} is a smooth and closed submanifold of the domain with no boundary of dimension $n-1$.
\item The restriction to the singular set of the original map in Definition \ref{def:1} is a smooth immersion.
\end{itemize}

\subsection{(Normal) crossings of a family of smooth immersions}
We define a {\it crossing} and a {\it normal} crossing of a family of smooth immersions.

Let $a>0$ be an integer and $\{c_j:X_j \rightarrow Y\}_{j=1}^{a}$ be a family of $a$ smooth immersions
from $m_j$-dimensional smooth manifolds $X_j$ without boundaries into an $n$-dimensional smooth manifold $Y$ with no boundary. A {\it crossing} of the family of the smooth immersions is a point $y \in Y$ such that
${\bigcup}_{j=1}^a {c_j}^{-1}(y)$ has at least two points. A crossing is said to be {\it normal} if the following two properties hold.
\begin{enumerate}
\item The disjoint union ${\bigcup}_{j=1}^a {c_j}^{-1}(y)$ of these preimages is a finite set consisting of exactly $b_y>1$ points.
\item Denote all the elements of ${\bigcup}_{j=1}^a {c_j}^{-1}(y)$ by $\{p_{y,j}\}_{j=1}^{b_y}$. Let $y(j)$ be the number satisfying $p_{y,j} \in X_{y(j)}$, which we can determine uniquely. We denote the dimension of the intersection
${\bigcap}_{j=1}^{b_y} {d {c_{y(j)}}}_{p_{y,j}}(T_{p_{y,j}}X_{y(j)})$ of the images of the differentials at all points $p_{y,j}$ by $y(c)$. In this situation, $y(c)+{\Sigma}_{j=1}^{b_y} (n-m_{y(j)})=n$.
\end{enumerate}

We can also consider this notion for a single immersion (the case $a=1$) similarly.

\begin{Def}
\label{def:1.2}
A {\it stable} fold map is a fold map whose restriction to the singular set is a smooth immersion such that the crossings of the restriction to the singular set of the original fold map are always normal.
\end{Def}

\begin{Rem}
Stable fold maps are actually defined as a fold map which is {\it stable}. However, we can also define as
Definition \ref{def:1.2}. For more precise and systematic explanations on {\it stable} (fold) maps, see
\cite{golubitskyguillemin} for example.
\end{Rem}

In the present paper, we consider crossings which are normal and preimages of which consist of at most two points.
\subsection{Reeb spaces}
Reeb spaces are also fundamental and important tools in investigating the topologies of the domains of smooth maps whose codimensions are negative. Let $X$ and $Y$ be topological spaces. For $p_1, p_2 \in X$ and for a continuous
map $c:X \rightarrow Y$, we define a relation ${\sim}_c$ on $X$ in the following way: $p_1 {\sim}_c p_2$ if and only if $p_1$ and $p_2$ are in a same connected component of $c^{-1}(p)$ for some $p \in Y$. Thus ${\sim}_{c}$ is an equivalence relation on $X$. We denote the quotient space $X/{\sim}_c$ by $W_c$.
\begin{Def}
\label{def:2}
We call $W_c$ the {\it Reeb space} of $c$.
\end{Def}
We denote the induced quotient map from $X$ into $W_c$ by $q_c$. We can define $\bar{c}: W_c \rightarrow Y$ in a unique way so that the relation $c=\bar{c} \circ q_c$ holds.

\begin{Prop}[\cite{shiota}]
\label{prop:1}
For {\rm (}stable{\rm )} fold maps, the Reeb spaces are polyhedra and the dimensions are equal to the dimensions of the target manifolds.
\end{Prop}

For Reeb spaces, see also \cite{reeb} for example.

\subsection{Explicit fold maps and their Reeb spaces}
We present fundamental and important examples of fold maps here. We explain terminologies on bundles.

For a topological space $X$, an {\it X-bundle} is a bundle whose fiber is $X$.

Hereafter, the structure groups of bundles such that their base spaces and fibers are manifolds are assumed to be (subgroups of) diffeomorphism groups except cases such as situations in a sketch of the proof of Proposition \ref{prop:3}: the bundles are {\it smooth} in a word except several cases. For an integer $k>0$, a {\it linear} bundle is a smooth bundle whose fiber is a $k$-dimensional unit disc (standard closed disc of a fixed diameter) or the ($k-1$)-dimensional unit sphere in ${\mathbb{R}}^{k}$ and whose structure group is a subgroup of the $k$-dimensional orthogonal group $O(k)$ acting linearly and canonically. For a point $p$ in a Eulidean space ${\mathbb{R}}^k$, $||p||$ denotes the distance between the origin $0 \in {\mathbb{R}}^k$ and $p$ where the underlying metric is the Eulidean metric.
\begin{Ex}
\label{ex:1}
\begin{enumerate}
\item
\label{ex:1.1}
A {\it special generic} map is defined as a fold map the index of each singular point of which is $0$.
Canonical projections of unit spheres are simplest stable special generic maps. According to studies \cite{saeki}, \cite{saeki2}, \cite{saekisakuma}, \cite{wrazidlo} and so on, homotopy spheres which are not diffeomorphic to standard spheres do not admit special generic maps into sufficiently high dimensional Euclidean spaces whose dimensions are smaller than those of the homotopy spheres. On the other hand, homotopy spheres except $4$-dimensional ones which are not diffeomorphic to $S^4$ admit special generic maps into the plane such that the restriction to the singular set is an embedding and that the singular set is a circle. Moreover, $4$-dimensional homotopy spheres which are not diffeomorphic to $S^4$ admit no special generic maps into the Euclidean spaces $\mathbb{R}$, ${\mathbb{R}}^2$ and ${\mathbb{R}}^3$. Note also that such homotopy spheres are still undiscovered.

The Reeb space of a (stable) special generic map $f$ from a closed and connected manifold of dimension $m$ into
${\mathbb{R}}^n$ satisfying the relation $m>n \geq 1$ is regarded as an $n$-dimensional compact and connected manifold we can immerse into ${\mathbb{R}}^n$. The image is regarded as the image of a suitable immersion of the $n$-dimensional manifold. The boundary of the Reeb space and the image of the singular set for $q_f$ agree. Conversely, for arbitrary integers $m>n \geq 1$ and an $n$-dimensional compact manifold we can immerse into ${\mathbb{R}}^n$, we can construct a (stable) special generic map from a suitable closed and connected manifold of dimension $m$ into ${\mathbb{R}}^n$ whose Reeb space is diffeomorphic to the $n$-dimensional manifold.

Moreover, we have the following bundles for a general special generic map $f$ from a closed and connected manifold of dimension $m$ into ${\mathbb{R}}^n$.

\begin{enumerate}
\item If we restrict the map $q_f$ to the preimage of the interior of the Reeb space, then it gives a smooth
$S^{m-n}$-bundle over the interior of the Reeb space.
\item If we restrict the map $q_f$ to the preimage of a small collar neighborhood of the boundary of the Reeb space and consider the composition of this with a canonical projection onto the boundary, then it gives a linear $S^{m-n+1}$-bundle over the boundary.
\end{enumerate}
Moreover, for an arbitrary $n$-dimensional compact manifold we can immerse into ${\mathbb{R}}^n$, we can construct a stable special generic map on a suitable closed and connected manifold of dimension $m$ into ${\mathbb{R}}^n$
so that the Reeb space is diffeomorphic to the $n$-dimensional manifold and that these bundles are trivial.
We can replace ${\mathbb{R}}^n$ by an arbitrary manifold $N$ of dimension $n$ with no boundary.
See \cite{saeki}, the articles \cite{kitazawa5} and \cite{kitazawa6} by the author, and so on.
\item
\label{ex:1.2}
(\cite{kitazawa}, \cite{kitazawa2} and \cite{kitazawa4}) Let $l>0$ be an integer.
Let $m>n \geq 1$ be integers. We can construct a stable fold map on a manifold represented as an $l-1$ connected sum of manifolds diffeomorphic to $S^{m-n} \times S^n$ into ${\mathbb{R}}^n$ (this is a standard sphere for $l=1$) satisfying the following properties.
\begin{enumerate}
\item The singular value set is represented as ${\sqcup}_{j=1}^{l} \{x \in {\mathbb{R}}^n \mid ||x||=j\}$.
\item Preimages of regular values are disjoint unions of standard spheres.
\item In the target space ${\mathbb{R}}^n$, the number of connected components of a preimage increases as we go straight to the origin $0 \in {\mathbb{R}}^n$ of the target Euclidean space starting from a point in the complement of the image.\end{enumerate}
The Reeb space is an $n$-dimensional polyhedron which is simple homotopy equivalent to a bouquet of $l-1$ copies of $n$-dimensional spheres for $l>1$ and a point for $l=1$.
We can construct a map satisfying these properties for manifolds obtained by changing the products to total spaces of general smooth $S^{m-n}$-bundles over $S^n$.
\item
\label{ex:1.3}
In \cite{kitazawa5}, \cite{kitazawa6}, \cite{kitazawa7}, and so on, present construction of stable fold maps satisfying the assumption of Proposition \ref{prop:3} in the last and information of the topologies such as the homology groups and the cohomology rings of the Reeb spaces and as a result we can know those of the manifolds. The maps are obtained by finite iterations of surgery operations ({\it bubbling operations}) starting from fundamental fold maps: this operation is introduced by the author in \cite{kitazawa5} respecting ideas of \cite{kobayashi}, \cite{kobayashi2} and \cite{kobayashisaeki} for example. More explicitly, we start from stable maps which arenotso complicated, by changing maps and manifolds by bearing new connected components of singular (value) sets of stable fold maps one after another, we obtain desired maps. The previous example are simplest examples explaining this well.
\end{enumerate}
\end{Ex}
\subsection{Construction of explicit fold maps by new surgery operations (bubbling operations) and the organization of this paper}
In the present paper, we present further studies on construction in Example \ref{ex:1} (\ref{ex:1.3}). We apply bubbling operations. These operations were first defined in \cite{kitazawa7} and improved in \cite{kitazawa5}.
The organization of the paper is as the following.
In the next section, we introduce a {\it bubbling operation} first introduced based on \cite{kitazawa7}. The last section is devoted to a presentation of new results Proposition \ref{prop:2}, Theorem \ref{thm:1} and Theorem \ref{thm:2}. We present construction of new families of explicit fold maps and investigate the cohomology rings of the Reeb spaces. We see that the cohomology rings of these newly obtained Reeb spaces are also obtained first in the present paper. Last we present Proposition \ref{prop:3}. This gives explicit situations where Reeb spaces of fold maps on manifolds know much topological information of the manifolds. We can apply this to new maps obtained by applying Proposition \ref{prop:2}, Theorem \ref{thm:1},Theorem \ref{thm:2}, and so on, explicitly. For Proposition \ref{prop:3}, an essentially same explanation is in the last of \cite{kitazawa7}.

\section{Bubbling operations and fold maps such that preimages of regular values are disjoint unions of spheres.}
\label{sec:2}




We introduce {\it bubbling operations}, first introduced in \cite{kitazawa5}, referring to \cite{kitazawa7}. More precisely, in the present paper, we essentially consider only the operations satisfying several good properties: {\it ATSS operations}.

Hereafter, $m>n \geq 1$ are integers, $M$ is a smooth, closed and connected manifold of dimension $m$, $N$ is a smooth manifold of dimension $n$ with no boundary, and $f:M \rightarrow N$ is a smooth map.

For a smooth map $c$, we denote the singular set $S(c)$.

\begin{Def}
\label{def:3}
For a stable fold map $f:M \rightarrow N$, let $P$ be a connected
component of $(W_f-q_f(S(f))) \bigcap {\bar{f}}^{-1}(N-f(S(f)))$, regarded as a manifold with no boundary diffeomorphic
to $\bar{f}(P) \subset N$.

Let $l>0$ and $l^{\prime} \geq 0$ be integers. Assume that there exist families $\{S_j\}_{j=1}^{l}$ of finitely many standard spheres and $\{N(S_j)\}_{j=1}^{l}$ of total spaces of linear bundles over these spheres with fibers being unit discs.
We also denote by $S_j$ the image of the section obtained by setting the value as the origin for each fiber diffeomorphic to a unit disc for each $N(S_j)$. Assume that the dimensions of $N(S_j)$ are always $n$. Assume also that there exist immersions $c_j:N(S_j) \rightarrow P$. Furthermore, the following properties hold.
\begin{enumerate}
\item $f {\mid}_{f^{-1}({\bigcup}_{j=1}^l c_j(N(S_j)))} f^{-1}({\bigcup}_{j=1}^l c_j(N(S_j))) \rightarrow {\bigcup}_{j=1}^l c_j(N(S_j))$ gives a trivial $S^{m-n}$-bundle.
\item Crossings of the family $\{{c_j} {\mid}_{\partial N(S_j)}:\partial N(S_j) \rightarrow P\}_{j=1}^l$ are normal and each preimage consists of exactly two points. The number of crossings of this family is finite. Furthermore, For any subset of the family consisting of exactly one immersion, there exists no crossing.
\item Crossings of $\{{c_j} {\mid}_{S_j}:S_j \rightarrow P\}_{j=1}^l$ are normal and each preimage consists of exactly two points. The number of crossings of this family is also finite.
\item We denote the set of all the crossings of the family $\{{c_j} {\mid}_{S_j}:S_j \rightarrow P\}_{j=1}^l$ of the immersions by $\{p_{j^{\prime}}\}_{j^{\prime}=1}^{l^{\prime}}$.
\item For each $p_{j^{\prime}}$ just before, there exist two integers $1 \leq a(j^{\prime}),b(j^{\prime}) \leq l$ and small standard closed discs $D_{2j^{\prime}-1} \subset  S_{a(j^{\prime})}$ and $D_{2j^{\prime}} \subset S_{b(j^{\prime})}$ satisfying the following properties.
\begin{enumerate}
\item $\dim D_{2j^{\prime}-1}=\dim S_{a(j^{\prime})}$.
\item $\dim D_{2j^{\prime}}=\dim S_{b(j^{\prime})}$.
\item $p_{j^{\prime}}$ is in the images of the immersions $p_{j^{\prime}} \in c_{a(j^{\prime})}({\rm Int } D_{2j^{\prime}-1})$ and $p_{j^{\prime}} \in c_{b(j^{\prime})}({\rm Int } D_{2j^{\prime}})$.
\item If $a(j^{\prime})=b(j^{\prime})$, then $D_{2j^{\prime}-1} \bigcap D_{2j^{\prime}}$ is empty.
\item If we restrict the base space of the bundle $N(S_{a(j^{\prime})})$ to $D_{2j^{\prime}-1}$ and that of
the bundle $N(S_{b(j^{\prime})})$ to $D_{2j^{\prime}}$, then the images of the total spaces of these resulting bundles by the immersions $c_{a(j^{\prime})}$ and $c_{b(j^{\prime})}$ agree as subsets in ${\mathbb{R}}^n$.
\item The restrictions to the total spaces of the two immersions just before are embeddings.
\item The set of all the crossings of the family of $\{{c_j} {\mid}_{S_j}:S_j \rightarrow P\}_{j=1}^l$ is the disjoint union of the $l^{\prime}$ corners of the subsets just before each of which is for $1 \leq j^{\prime} \leq l^{\prime}$.
\end{enumerate}
\end{enumerate}
In this situation, the family $\{(S_j,N(S_j),c_j:N(S_j) \rightarrow P)\}_{j=1}^{l}$ is said to be a {\it normal system of submanifolds} compatible with $f$.
\end{Def}
The second property ofr the five listed properties is a bit different from the original one.
In the situation of Definition \ref{def:3}, let $\{N^{\prime}(S_j) \subset N(S_j)\}_{j=1}^{l}$ be a family of total spaces of subbundles of $\{N(S_j)\}_{j=1}^{l}$ over the manifolds whose fibers are standard closed discs. We assume that the diameters are all $0<r<1$. For a suitable $r$, same properties as presented in Definition \ref{def:3} hold. In other words, we
can obtain another family $\{(S_j,N^{\prime}(S_j),{c_j} {\mid}_{N^{\prime}(S_j)}:N^{\prime}(S_j) \rightarrow P)\}_{j=1}^l$ and this is also regarded as a normal system of submanifolds compatible with $f$: we can identify each fiber, which is a standard closed disc of diameter $r$ with a unit disc via the diffeomorphism defined by the correspondence $t \mapsto \frac{1}{r} t$.

\begin{Def}
\label{def:4}
\begin{enumerate}
\item The familiy $\{(S_j,N(S_j),c_j:N(S_j) \rightarrow P)\}_{j=1}^{l}$ is said to be a {\it wider normal system supporting} the normal system of submanifolds $\{(S_j,N^{\prime}(S_j),{c_j} {\mid}_{N^{\prime}(S_j)}:N^{\prime}(S_j) \rightarrow P)\}_{j=1}^{l}$ compatible with $f$.
\item For a stable fold map $f:M \rightarrow N$ and an integer $l>0$, let $P$ be a connected
component of $(W_f-q_f(S(f))) \bigcap {\bar{f}}^{-1}(N-f(S(f)))$ and let $\{(S_j,N(S_j),c_j:N(S_j) \rightarrow P)\}_{j=1}^{l}$ be a normal system of submanifolds compatible with $f$.
Let $\{(S_j,N^{\prime}(S_j),{c_j}^{\prime}:N^{\prime}(S_j) \rightarrow P)\}_{j=1}^{l}$ be a wider normal system supporting this as before.
Assume that we can construct a stable fold map $f^{\prime}$ on an $m$-dimensional closed manifold $M^{\prime}$
into ${\mathbb{R}}^n$ satisfying the following properties.
\begin{enumerate}
\item $Q$ is the preimage $f^{-1}({\bigcup}_{j=1}^l {c_j}^{\prime}(N^{\prime}(S_j)))$.
\item $M-{\rm Int} Q$ is a compact submanifold of $M^{\prime}$ of dimension $m$ where a suitable smooth embedding $e:M-{\rm Int} Q \rightarrow M^{\prime}$ is defined.
\item $f {\mid}_{M-{\rm Int} Q}={f }^{\prime} \circ e {\mid}_{M-{\rm Int} Q}$ holds.
\item ${f}^{\prime}(S({f}^{\prime}))$ is the disjoint union of $f(S(f))$ and ${\bigcup}_{j=1}^n c_j(\partial N(S_j))$.
\item The indices of points in the preimages of new connected components in the resulting singular value set are always $1$.
\item For each regular value $p$ of the resulting map sufficiently close to the union ${\bigcup}_{j=1}^l c_j(S_j)$, the preimages are disjoint unions of standard spheres.
\end{enumerate}
This yields a procedure of constructing $f^{\prime}$ from $f$. We call it an {\it ATSS operation} to $f$.
The union ${\bigcup}_{j=1}^l c_j(S_j)$ is called the {\it generating normal system} of the operation.
\end{enumerate}
\end{Def}

We show a local fold map around $p_{j^{\prime}}$ in Definition \ref{def:3}. This is also presented in \cite{kitazawa7} with FIGURE 2. $S^{0}$ is the two point set with the discrete topology.

First in the situation of Example \ref{ex:1} (\ref{ex:1.2}), we explain a restriction of a fold map for $l=1$ to the preimage of the set of all points $x \in {\mathbb{R}}^n$ satisfying $||x|| \leq \frac{3}{2}$. We construct a product
bundle $D^n \times S^{m-n}$ over $D^n$. We also set $D^n$ as the set of all points $x \in {\mathbb{R}}^n$
satisfying $||x|| \leq \frac{1}{2}$.
We also set a Morse function ${\tilde{f}}_{m-n,0}$ on a manifold obtained by removing the interior of a standard closed disc of dimension $m-n+1$ embedded smoothly in the interior of $S^{m-n} \times [-1,1]$
onto $[\frac{1}{2},\frac{3}{2}] \subset (0,+\infty) \subset \mathbb{R}$ such that the following four hold.
\begin{enumerate}
\item The preimage of the minimum coincides with the disjoint union of two connected components of the boundary.
\item The preimage of the maximum coincides with one connected component of the boundary.
\item There exists exactly one singular point, and the singular point is in the interior.
\item The value at the singular point is $1$.
\end{enumerate}
We glue the projection of the product bundle and the map
${\tilde{f}}_{m-n,0} \times {\rm id}_{S^{n-1}}:[\frac{1}{2},+\infty) \times S^{n-1}$ where we identify the base space $D^n$ of the product bundle with a standard closed disc of dimension $n$ whose diameter is $\frac{1}{2}$ and whose center is the origin $0 \in {\mathbb{R}}^n$.
By gluing suitably, we have a desired smooth map onto the standard closed disc of dimension $n$ whose center is the origin and whose diameter is $\frac{3}{2}$ in ${\mathbb{R}}^n$. This is a desired map, obtained by restricting the original fold map for $l=1$.
See also \cite{kitazawa} and \cite{kitazawa4} for the fold map for $l=1$.

We denote the resulting map onto the standard closed disc of dimension $n$ by ${\tilde{f}}_{m,n,0}$.

We consider the composition of ${\tilde{f}}_{m-n+\dim D_{2j^{\prime}-1},\dim D_{2j^{\prime}-1},0}$ with a suitable diffeomorphism, we have a smooth map onto a sufficiently small standard closed disc
${D^{\prime}}_{2j^{\prime}-1} \supset D_{2j^{\prime}-1}$ of dimension $\dim D_{2j^{\prime}-1}$ satisfying
$S_{a(j^{\prime})} \supset {D^{\prime}}_{2j^{\prime}-1} \supset {\rm Int} {D^{\prime}}_{2j^{\prime}-1} \supset D_{2j^{\prime}-1}$. We can take a sufficiently small standard closed disc ${D^{\prime}}_{2j^{\prime}} \supset D_{2j^{\prime}}$ of dimension
$\dim D_{2j^{\prime}}$ satisfying similar properties.
We can consider the product map of the previous smooth map and the identity map ${\rm id}_{{D^{\prime}}_{2j^{\prime}}}$. We compose the resulting map with a suitable diffeomorphism onto a manifold, regarded as
$c_{a(j^{\prime})}({D^{\prime}}_{2j^{\prime}-1}) \times c_{b(j^{\prime})}({D^{\prime}}_{2j^{\prime}})$ and containing
$p_{j^{\prime}} \in {\mathbb{R}}^n$ in its interior.

We can restrict the map to a total space of a trivial $D^{m-n}$-bundle over the target space, regarded as
$c_{a(j^{\prime})}({D^{\prime}}_{2j^{\prime}-1}) \times c_{b(j^{\prime})}({D^{\prime}}_{2j^{\prime}})$ and containing
$p_{j^{\prime}}$ in the interior.

For the composition of ${\tilde{f}}_{m-n+\dim D_{2j^{\prime}},\dim D_{2j^{\prime}},0}$ with a suitable diffeomorphism, we can restrict the map to a total space of a trivial $D^{m-n}$-bundle over the target space, diffeomorphic to a standard closed disc of dimension $\dim D_{2j^{\prime}}$. This gives a trivial $D^{m-n}$-bundle. The total space is regarded as a submanifold of the domain of the original map and we can also restrict the composition of
${\tilde{f}}_{m-n+\dim D_{2j^{\prime}},\dim D_{2j^{\prime}},0}$ with the suitable diffeomorphism to the closure of its complement in the domain. The closure of the complement is also a compact submanifold of dimension
$m-n+\dim D_{2j^{\prime}}$. We consider the composition of the product map of the restriction to this complement and the identity map ${\rm id}_{{D^{\prime}}_{2j^{\prime}-1}}$ with a suitable diffeomorphism onto a manifold, regarded as
$c_{a(j^{\prime})}({D^{\prime}}_{2j^{\prime}-1}) \times c_{b(j^{\prime})}({D^{\prime}}_{2j^{\prime}})$.

In a suitable way, we replace the original projection of the trivial $D^{m-n}$-bundle over the target space, identified with $c_{a(j^{\prime})}({D^{\prime}}_{2j^{\prime}-1}) \times c_{b(j^{\prime})}({D^{\prime}}_{2j^{\prime}})$, containing $p_{j^{\prime}}$ in the interior, by the map just before: the composition of the product map of the restriction to the complement of a total space of a trivial $D^{m-n}$-bundle over the target space, diffeomorphic to a standard closed disc of dimension
$\dim D_{2j^{\prime}}$, and the identity map ${\rm id}_{{D^{\prime}}_{2j^{\prime}-1}}$ with a suitable diffeomorphism onto a manifold, regarded as
$c_{a(j^{\prime})}({D^{\prime}}_{2j^{\prime}-1}) \times c_{b(j^{\prime})}({D^{\prime}}_{2j^{\prime}})$. The resulting map is a desired local map and said to be a {\it local canonical fold map around a crossing} for an ATSS operation.

By the definition, the following corollary immediately follows.
\begin{Cor}
\label{cor:1}
Let $f:M \rightarrow N$ be a stable fold map on an $m$-dimensional closed and connected manifold $M$ into an
$n$-dimensional manifold $N$ with no boundary satisfying $m-n>1$. If an ATSS operation is performed to $f$ and a new map $f^{\prime}$ is obtained as a result, then $W_f$ is a proper subset of $W_{{f}^{\prime}}$ and
$\bar{{f}^{\prime}} {\mid}_{W_f}=\bar{f}:W_f \rightarrow N$.
\end{Cor}

\section{New families of stable fold maps and their Reeb spaces obtained by ATSS operations increasing the numbers of connected components of singular sets by two.}
\label{sec:3}
In this section, we need fundamental theory of graded commutative algebras over commutative rings (PIDs) and cohomology rings.

\begin{Def}
\label{def:5}
Let $A$ be a module over a commutative ring $R$ having a unique identity element $1 \neq 0 \in R$.
Let $a \in A$ be a non-zero element such that the following two hold.
\begin{itemize}
\item $a$ is not represented as $a=ra^{\prime}$ where $(r,a^{\prime}) \in R \times A$ and $r$ is not a unit.
\item For $r \in R$, $ra=0$ if and only if $r=0$.
\end{itemize}  
Let $A$ be represented as an inner direct sum of the submodule generated by the one element set $\{a\}$ and another submodule $B$. We can define a homomorphism $a^{\ast}:A \rightarrow R$ between the modules over $R$ such that $a^{\ast}(a)=1$ and that $a^{\ast}(B)=0$ for any such $B$. We call $a^{\prime}$ the {\it dual} of $a \in A$.
\end{Def}

For a graded commutative algebra $A$ over a graded commutative ring $R$ and a non-negative integer $i \geq 0$, we call the module of all elements of degree $i$ the {\it $i$-th module} of $A$. Hereafter, we assume that the $0$-th module is $R$ and equipped with a canonical action by $R$.

\begin{Def}
\label{def:6}
Let $A_1$ and $A_2$ be graded commutative algebras over a commutative ring $R$ having a unique identity element $1 \neq 0 \in R$.
Let $A$ be a graded commutative algebra over $R$ satisfying the following properties.
\begin{enumerate}
\item The $i$-th module is the direct sum of the $i$-th modules of $A_1$ and $A_2$. 
\item Let $i_1,i_2>0$ be integers.
For $(a_{i_1,1},a_{i_1,2}), (a_{i_2,1},a_{i_2,2}) \in A_1 \oplus A_2$, which are elements of $i_1$-th and $i_2$-th modules, respectively, the product is $(a_{i_1,1}a_{i_2,1}, a_{i_1,2}a_{i_2,2}) \in A_1 \oplus A_2$.
\item For $r \in R$, where we define this as an element of the $0$-th module and take the product of this and an element $(a_{i,1},a_{i,2}) \in A_1 \oplus A_2$ of the $i$-th module where $i>0$. The product is $(ra_{i,1},ra_{i,2}) \in A_1 \oplus A_2$.  
\end{enumerate}
$A$ is called a {\it graded commutative algebra obtained canonically from the direct sum} $A_1 \oplus A_2$.
\end{Def}

\begin{Def}
\label{def:7}
Let $X$ be a topological space regarded as a polyhedron (e. g. a smooth manifold is in a canonical way regarded as a polyhedron). A homology class $c \in H_j(X;A)$ is {\it represented} by a map  $c_{Y,X}:Y \rightarrow X$ regarded as a PL map from a topological space $Y$ regarded as a polyhedron of dimension $j$ if the value of ${c_{Y,X}}_{\ast}$ at the homology class represented by a cycle canonically obtained from $Y$ and the class $c$ agree where the coefficient is a commutative group $A$. Especially, $c$ is {\it represented} by a subspace $Y$ (, regarded as a subpolyhedron,) means that in this situation $c_{Y,X}:Y \rightarrow X$ is the canonical embedding. 
\end{Def}

We first obtain examples as Proposition \ref{prop:2}.
\begin{Prop}
\label{prop:2}
Let $R$ be a PID having an identity element $1 \in R$ satisfying $1 \neq 0 \in R$. Let $t \in R$ be an element represented as $t_{0} \in \mathbb{Z}$ times the identity element $1$ for $i=1,2$.
Let $f:M \rightarrow N$ be a stable fold map on an $m$-dimensional closed and connected manifold into an $n$-dimensional connected and open manifold satisfying $m-n>1$. We also assume at least one of the following conditions.

Let $U$ be an open set in $N-f(S(f))$ such that $f {\mid}_{f^{-1}(U)}:f^{-1}(U) \rightarrow U$ gives a trivial $S^{m-n}$-bundle.

In this situation, by an ATSS operation to $f$, we have a new fold map $f^{\prime}$ satisfying the following properties if $n$ is even.

\begin{enumerate}
\item $H_i(W_{f^{\prime}};R)$ is isomorphic to $H_i(W_f;R)$ for $i \neq \frac{n}{2}$ and $H_i(W_f;R) \oplus R \oplus R$ for $i=\frac{n}{2},n$.
\item The cohomology group $H^{i}(W_{f^{\prime}};R)$ is isomorphic to $H^i(W_f;R)$ for $i \neq \frac{n}{2}$ and $H^i(W_f;R) \oplus R \oplus R$ for $i=\frac{n}{2},n$: we can set isomorphisms between modules over $R$ for identifications respecting Definition \ref{def:6} and we will abuse these identifications. The cohomology group $H^{i}(W_{f};R)$ is isomorphic to $H^i(W_{f^{\prime}};R)$ for $i \neq \frac{n}{2},n$ and identified with $H^i(W_f;R) \oplus \{0\} \oplus \{0\}$ in $H^i(W_f;R) \oplus R \oplus R$ before for $i=\frac{n}{2},n$ via a map $i_{f,f^{\prime}}(x):= (x,0,0)$ for $x \in H^{i}(W_f;R)$. Furthermore this gives a monomorphism over $R$ from $H^{i}(W_{f};R)$ into $H^{i}(W_{f^{\prime}};R)$ for all $i$: for $i \neq \frac{n}{2},n$, $i_{f,f^{\prime}}(x)=x$ for $x \in H^{i}(W_{f};R)$, which we abbreviated just before.
\item We can define a graded commutative algebra $A_R$ over $R$ such that the $i$-th module is isomorphic to $\{0\}$ for $i \neq 0,\frac{n}{2},n$ and isomorphic to $R \oplus R$ and identified with $\{0\} \oplus R \oplus R$ in $H^i(W_f;R) \oplus R \oplus R$ before for $i=\frac{n}{2},n$ and that the following rules are satisfied.
\begin{enumerate}
\item Under the explained identifications, the product of $(0,1,0) \in \{0\} \oplus R \oplus R \subset H^{\frac{n}{2}}(W_f;R) \oplus R \oplus R$ and $(0,0,1) \in \{0\} \oplus R \oplus R \subset H^{\frac{n}{2}}(W_f;R) \oplus R \oplus R$ is $(0,t,t) \in \{0\} \oplus R \oplus R \subset H^{n}(W_f;R) \oplus R \oplus R$.
\item Under the explained identifications, the product of $(0,1,0) \in \{0\} \oplus R \oplus R \subset H^{\frac{n}{2}}(W_f;R) \oplus R \oplus R$ and $(0,1,0) \in \{0\} \oplus R \oplus R \subset H^{\frac{n}{2}}(W_f;R) \oplus R \oplus R$ is $(0,0,0) \in \{0\} \oplus R \oplus R \subset H^{n}(W_f;R) \oplus R \oplus R$.
\item Under the explained identifications, the product of $(0,0,1) \in \{0\} \oplus R \oplus R \subset H^{\frac{n}{2}}(W_f;R) \oplus R \oplus R$ and $(0,0,1) \in \{0\} \oplus R \oplus R \subset H^{\frac{n}{2}}(W_f;R) \oplus R \oplus R$ is $(0,0,0) \in \{0\} \oplus R \oplus R \subset H^{n}(W_f;R) \oplus R \oplus R$ 
\end{enumerate}
\end{enumerate}
We also have a new fold map $f^{\prime}$ satisfying the following properties if $n$ is an arbitrary positive integer and $k$ is a positive integer $2k<n$.
\begin{enumerate}
\item $H_i(W_{f^{\prime}};R)$ is isomorphic to $H_i(W_f;R)$ for $i \neq k,n-k,n$, $H_i(W_f;R) \oplus R$ for $i=k,n-k$ and $H_n(W_f;R) \oplus R \oplus R$ for $i=n$.
\item The cohomology group $H^{i}(W_{f^{\prime}};R)$ is isomorphic to $H^i(W_f;R)$ for $i \neq k,n-k,n$, $H^i(W_f;R) \oplus R$ for $i=k,n-k$ and $H^i(W_f;R) \oplus R \oplus R$ for $i=n$: we can set isomorphisms between modules over $R$ for identifications respecting Definition \ref{def:6}. The cohomology group $H^{i}(W_{f};R)$ is isomorphic to $H^i(W_{{f}^{\prime}};R)$ for $i \neq k,n-k,n$, identified with $H^i(W_f;R) \oplus \{0\}$ in $H^i(W_f;R) \oplus R$ before for $i=k,n-k$ via a map $i_{f,f^{\prime}}(x):= (x,0)$ for $x \in H^{i}(W_f;R)$ and identified with $H^i(W_f;R) \oplus \{0\} \oplus \{0\}$ in $H^i(W_f;R) \oplus R \oplus R$ before for $i=n$ via a map $i_{f,f^{\prime}}(x):= (x,0,0)$ for $x \in H^{i}(W_f;R)$. Furthermore the maps give a monomorphism over $R$ from $H^{i}(W_{f};R)$ into $H^{i}(W_{f^{\prime}};R)$ for all $i$: for $i \neq k,n-k,n$, $i_{f,f^{\prime}}(x)=x$ for $x \in H^{i}(W_{f};R)$, which we abbreviated just before.
\item We can define a graded commutative algebra $A_R$ over $R$ such that the $i$-th module is isomorphic to $\{0\}$ for $i \neq 0,k,n-k,n$, isomorphic to $R$ and identified with $\{0\} \oplus R$ in $H^i(W_f;R) \oplus R$ before for $i=k,n-k$ and isomorphic to $R \oplus R$ and identified with $\{0\} \oplus R \oplus R$ in $H^i(W_f;R) \oplus R \oplus R$ before for $n$ and that the following rules are satisfied.
\begin{enumerate}
\item Under the explained identifications, the product of $(0,1) \in \{0\} \oplus R \subset H^{k}(W_f;R) \oplus R$ and $(0,1) \in \{0\} \oplus R \subset H^{n-k}(W_f;R) \oplus R$ is $(0,t,t) \in \{0\} \oplus R \oplus R \subset H^{n}(W_f;R) \oplus R \oplus R$.
\item Under the explained identifications, the square of $(0,1) \in \{0\} \oplus R \subset H^{k}(W_f;R) \oplus R$ vanishes.
\item Under the explained identifications, the square of $(0,1) \in \{0\} \oplus R \subset H^{n-k}(W_f;R) \oplus R$ vanishes. 
\end{enumerate}
\end{enumerate}

Last the resulting cohomology ring $H^{\ast}(W_{f};R)$ is regarded as a subalgebra of $H^{\ast}(W_{f^{\prime}};R)$ via the following rules.
\begin{enumerate}
\item For two elements $c_1 \in H^{i_1}(W_{f};R)$ and $c_2 \in H^{i_2}(W_{f};R)$ where $i_1,i_2>0$, we consider the natural identifications so that the elements are regarded as $(c_j,0) \in H^{i_j}(W_{f};R) \oplus \{0\} \subset H^{i_j}(W_{f^{\prime}};R)$ for $j=1,2$ and the product is $(c_1c_2,0) \in H^{i_1+i_2}(W_{f};R) \oplus \{0\} \subset H^{i_1+i_2}(W_{f^{\prime}};R)$: we can identify this with $c_1c_2 \in H^{i_1+i_2}(W_{f};R)$. 
\item For elements $r \in H^{0}(W_{f};R)$ and $c \in H^{i}(W_{f};R)$ where $i>0$, we consider the natural identifications so that they are regarded as $r \in H^{0}(W_{f};R)$ where the group is identified with $H^{0}(W_{f^{\prime}};R)$ and $(c,0) \in H^{i}(W_{f};R) \oplus \{0\} \subset H^{i}(W_{f^{\prime}};R)$, respectively, and the product is $(rc,0) \in H^{i}(W_{f};R) \oplus \{0\} \subset H^{i}(W_{f^{\prime}};R)$: we can identify this with $rc \in H^{i}(W_{f};R)$. 
\end{enumerate}
\end{Prop}
We prove this using tools and ideas used for the proof of Proposition 3 of \cite{kitazawa7}.  
\begin{proof}
In the proof, notation in Definition \ref{def:3} and around this will be used. Let us find a suitable normal system of submanifolds compatible with $f$. As this, we will find $\{(S_1,N(S_1),c_1:N(S_1) \rightarrow P), (S_2,N(S_2),c_2:N(S_2) \rightarrow P)\}$ such that $S_1$ is a standard sphere of dimension $0<k \leq \frac{n}{2}$ and that $S_2$ is standard sphere of dimension $n-k$.

We can take $\{(S_1,N(S_1),c_1:N(S_1) \rightarrow P), (S_2,N(S_2),c_2:N(S_2) \rightarrow P)\}$ so that the family of immersions has exactly $|t_{0}| \geq 0$ pairs of crossings ($2|t_{0}|$ crossings), that the normal bundle of the immersion is trivial and that $c_1(N(S_1)) \bigcup c_2(N(S_2)) \subset U$.
We can perform an ATSS operation whose generating normal system is $c_1(S_1) \bigcup c_2(S_2)$ to obtain a new fold map $f^{\prime}$. We use local canonical fold maps around crossings in section \ref{sec:2} around crossings in $c_1(S_1) \bigcup c_2(S_2)$. Around the remaining singular values and regular values, we construct products of Morse functions with exactly one singular point, which is of index $1$, and identity maps on ($n-1$)-dimensional manifolds and trivial $S^{m-n}$-bundles over $n$-dimensional manifolds, respectively. We can glue all the local maps together. We can perform the construction thanks to the assumption that $f {\mid}_{f^{-1}(U)}:f^{-1}(U) \rightarrow U$ gives a trivial $S^{m-n}$-bundle.

By the definition of a normal bubbling operation, $W_{f^{\prime}}$ is regarded as a space obtained by attaching a polyhedron $A$ we can obtain by identifying exactly $|l_{0}|$ pairs of disjointly embedded PL discs of dimension $n$, represented as a product of of a $k$-dimensional disc and an ($n-k$)-dimensional discs, two smooth (PL) manifolds $S_1 \times S^{n-k}$ and $S^2 \times S^k$ to $B:={\bar{f}}^{-1}(c_1(N(S_1))) \bigcup {\bar{f}}^{-1}(c_2(N(S_2))) \subset W_f$. Moreover, $S_1 \times D^{n-k}$ and $S^2 \times D^k$ are attached: $D^{n-k} \subset S^{n-k}$ and $D^k \subset S^k$ are hemispheres.  

%



We explain the topologies of the Reeb spaces and the polyhedra without using sophisticated terminologies on Mayer-Vietoris sequences and other notions of algebraic topology. Rigorous understandings via these terminologies are left to readers and see also Proposition 3 of \cite{kitazawa7}.

$H_i(W_{f^{\prime}};R)$ is isomorphic to $H_i(W_f;R)$ for $i \neq k$, $H_i(W_f;R) \oplus R$ for $i=k \neq \frac{n}{2}$, $H_i(W_f;R) \oplus R \oplus R$ for $i=k=\frac{n}{2}$ and $H_i(W_f;R) \oplus R$ for $i=n-k \neq \frac{n}{2}$. We can identify the modules for these cases. 

We explain about the summands $R$. For $i=k,n-k$, the summands $R$ are seen to be generated by the class represented by $\{\ast\} \times S^{n-k} \subset S_1 \times S^{n-k}$ or $\{\ast\} \times S^{k} \subset S_2 \times S^{k}$ in the original manifolds to obtain $A$ where $\ast$ is a suitable point in $S_1$ or $S_2$. In the case $i=n$, the summands $R$ are seen to be generated by the classes represented by suitable subpolyhedra of $A$ obtained after the deformation of the original manifolds to obtain $A$, which are $n$-dimensional polyhedra obtained from the original $n$-dimensional closed, connected and orientable manifolds.

We discuss the cohomology rings. By the construction, the resulting cohomology ring is isomorphic to and can be identified with a graded commutative algebra obtained canonically from the direct sum of $H^{\ast}(W_f;R)$ and a new graded commutative algebra $A_R$. We denote the $j$-th module of $A_R$ by $A_{R,j}$. This is zero unless $j=k,n-k,n$. 
Both in the cases where $k \neq n-k$ and $k=\frac{n}{2}$, we consider the products in $H^n(W_{f^{\prime}};R)$ of the two classes in $H^k(W_{f^{\prime};R})$ and $H^{n-k}(W_{f^{\prime};R})$ generating the summands $R$.
$A_{R,i}$ is isomorphic to and identified with $R$ for $i=k,n-k$ and $k \neq \frac{n}{2}$. $A_{R,i}$ is isomorphic to and identified with $R \oplus R$ for an even $n$ and $i=\frac{n}{2}$. $A_{R,n}$ is isomorphic to and identified with $R \oplus R$.

We note that $H^i(W_{f^{\prime}};R)$ is isomorphic to $H^i(W_f;R)$ for $i \neq k$, $H^i(W_f;R) \oplus R$ for $i=k \neq \frac{n}{2}$, $H^i(W_f;R) \oplus R \oplus R$ for $i=k=\frac{n}{2}$ and $H^i(W_f;R) \oplus R$ for $i=n-k \neq \frac{n}{2}$ and that we can identify the modules for these cases respecting Definition \ref{def:6} as explained.

For $k \neq \frac{n}{2}$, consider the product of $1 \in R$, identified with $A_{R,k}$, and $1 \in R$, identified with $A_{R,n-k}$ and for $k = \frac{n}{2}$, consider the product of $(1,0) \in R \oplus R$ and $(0,1) \in R \oplus R$, where the modules are identified with $A_{R,k}$

In the case where the number $t_{0}$ is $0$, the product vanishes (see also the proofs of some propositions and theorems of \cite{kitazawa6}). In the case where the number $t_{0}$ is $1$, the product can be $(1,1) \in R \oplus R$, identified with $A_{R,n}$. Let us explain this more precisely. 

We can define cohomology classes regarded as the duals of the homology classes represented by $\{{\ast}_2\} \times S^{k} \subset S_2 \times S^{k}$ and $\{{\ast}_1\} \times S^{n-k} \subset S_1 \times S^{n-k}$ before: the value of the dual at the original homology class is the identity element $1 \in R$, the value at the remaining class of the two classes is zero, and the values at classes in $H_{i}(W_f;R) \oplus \{0\}$ are zero where natural identifications with $H_i(W_{f^{\prime}};R)$ before are considered ($i=k,n-k$). We evaluate the value of the product at the classes represented by the $n$-dimensional polyhedra obtained from the products of two standard spheres in $A$: more precisely, the classes obtained after the manifolds are deformed and attached to $W_f$. For the original manifolds $S_1 \times S^{n-k}$ and $S_2 \times S^{k}$, the class represented by $S_1 \times \{{\ast}^{\prime}\} \subset S_1 \times S^{n-k}$ can be mapped to the class represented by $\{{\ast}^{\prime \prime}\} \times S^{k}$ in the other original manifold deformed and attached to and regarded as a subspace in $W_{f^{\prime}}$. It also can be mapped to zero if we perform an ATSS operation in a suitable way. The class represented by $S_2 \times \{{\ast}^{\prime \prime \prime}\} \subset S_2 \times S^{k}$ can be mapped to zero.

We give an additional explanation on this using FIGURE 3 of \cite{kitazawa7}.

We consider one point in the pair $(p_1,p_2)$ of the crossing in $c_1(S_1)$: take $p_2$. We can take $D_{1} \subset S_1$, $D_{2} \subset S_2$, $D_{3} \subset S_1$ and $D_{4} \subset S_2$ as in Definition \ref{def:3}.
Dots denote $S_1 \times \{{{\ast}_1}^{\prime}\} \subset S_1 \times S^{n-k}$ and $S_2 \times \{{{\ast}_2}^{\prime}\} \subset S_2 \times S^{k}$ deformed and attached to the Reeb space. We represent the original manifold $S_1 \times S^{n-k}$ by $S_1 \times ({D^{n-k}}_1 \bigcup  {D^{n-k}}_2)$ where ${D^{n-k}}_1$ and ${D^{n-k}}_2$ are hemispheres and the boundaries are an equator of $S^{n-k}$. We can perform an ATSS operation so that $D_2$ and $\{q_1\} \times {D^{n-k}}_1 \subset D_1 \times S^{n-k}$ agree in $W_{f^{\prime}}$ for a suitable point $q_1$. We can perform an ATSS operation so that $D_4$ and $\{q_2\} \times {D^{n-k}}_1 \subset D_3 \times S^{n-k}$ agree and also perform an ATSS operation so that $D_4$ and $\{q_2\} \times {D^{n-k}}_2 \subset D_3 \times S^{n-k}$ agree in $W_{f^{\prime}}$ for a suitable point $q_2$. We can do so that either of the following two holds in $W_{f^{\prime}}$.
\begin{enumerate}
\item $S_1 \times \{{{\ast}_1}^{\prime}\} \subset S_1 \times S^{n-k}$ is mapped so that the map is homotopic to the embedding into a fiber of the trivial bundle $S_2 \times S^{k}$ over $S^2$.  $S_2 \times \{{{\ast}_2}^{\prime}\} \subset S_2 \times S^{k}$ is mapped so that the map is homotopic to the embedding into a fiber of the trivial bundle $S_1 \times S^{n-k}$ over $S^2$. 
\item $S_1 \times \{{{\ast}_1}^{\prime}\} \subset S_1 \times S^{n-k}$ is mapped so that the map is homotopic to a constant map. $S_2 \times \{{{\ast}_2}^{\prime}\} \subset S_2 \times S^{k}$ is mapped so that the map is homotopic to a constant map.
\end{enumerate}

The argument yields the fact that the values of product before at the classes represented by the two $n$-dimensional polyhedra or the classes obtained after the original manifolds, which are diffeomorphic to products of two spheres, are deformed and attached to $W_f$, are both $0$ or both $1$. As a result, the product is represented as $(1,1) \in R \oplus R$, identified with $\{0\} \oplus R \oplus R$ and $\{0\} \oplus A_{R,n} \subset H^n(W_{f^{\prime}};R)$ (under the suitable identification of $H^n(W_f;R) \oplus A_{R,n}$ with $H^n(W_{f^{\prime}};R)$) or the product vanishes.

For a general integer $t_0$ and $t \in R$, we can argue similarly.
Moreover, on the square of each of the $k$-th and ($n-k$)-th cohomology classes, we can see the vanishing easily observing the topologies of the Reeb spaces. 

By the construction via the ATSS operation, we can easily see that last statement holds: $H^{\ast}(W_f;R)$ is regarded as a subalgebra of $H^{\ast}(W_{f^{\prime}};R)$.

This completes the proof.

\end{proof}
We introduce a {\it connected sum} of two stable fold maps. 

For integers $m>n \geq 1$, let $M_i$ ($i=1,2$) be a closed and connected manifold of dimension $m$ and $f_1:M_1 \rightarrow N$ be a stable fold map such $N-f_1(M_1)$ is not empty
and $f_{1.5},:M_2 \rightarrow {\mathbb{R}}^n$ be a stable fold map.

Let $e_1 \in {\mathbb{R}}^n$ be the point such that the first component is $1$ and that the remaining components are all $0$. Let ${D}_{1,1}^{n}:=\{x=(x_1,\cdots,x_{n}) \in {\mathbb{R}}^{n} \mid ||x-e_1|| \leq 1\}$.
We consider the canonical projection of a unit sphere $S^m \subset {\mathbb{R}}^{m+1}$ to ${\mathbb{R}}^n$ defined as the composition of the canonical inclusion with the projection ${\pi}_{m+1,n}((x_1, \cdots, x_n, \cdots, x_{m+1}))=(x_1, \cdots, x_n)$ and the restriction to the preimage of ${D}_{1,1}^{n}$: we denote the restriction by ${\pi}_{m,n,SB}$. 

We consider a composition of $f_{1.5}$ with an embedding $e:{\mathbb{R}}^n \rightarrow S^n$ and we denote the resulting map by $f_2$. We can take a standard closed disc $P_i$ of dimension $n$ such that there exists a pair $({\Phi}_i,{\phi}_i)$ of diffeomorphisms satisfying the relation ${\phi}_i \circ f_i {\mid}_{{f_i}^{-1}(P_i)}={\pi}_{m,n,SB} \circ {\Phi}_i$ for $i=1,2$ (the target of the left map is $P_i$ and that of the right map is ${D}_{1,1}^{n}$). 

We set $N_1:=N$ and $N_2:=S^n$.
 We can glue the maps ${f_i} {\mid}_{{f_i}^{-1}(N_i-{\rm Int} P_i)}:{f_i}^{-1}(N_i-{\rm Int} P_i) \rightarrow N_i-{\rm Int} P_i$ ($i=1,2$) on the boundaries to obtain a new map and by composing a diffeomorphism from the new target space to $N_1=N$, we obtain a smooth map into $N_1=N$ so that the resulting domain is represented as a connected sum of the original manifolds $M_1$ and $M_2$. The resulting fold map is a {\it connected sum} of $f_1$ and $f_{1.5}$. For $N:={\mathbb{R}}^n$, a connected sum is essentially equivalent to one presented in \cite{kitazawa5}, \cite{kitazawa6}, and so on.

Proposition \ref{prop:2} yields the following main theorem.

\begin{Thm}
\label{thm:1}
Let $R$ be a PID having an identity element $1 \in R$ satisfying $1 \neq 0 \in R$. 
Let $f:M \rightarrow N$ be a stable fold map on an $m$-dimensional closed and connected manifold into an $n$-dimensional connected and open manifold satisfying $m-n>1$.
Let $k>0$ be an integer satisfying $2k \leq n$. Let $t \in R$ be an element represented as $t_{0} \in \mathbb{Z}$ times the identity element $1$.
Let $A_R$ be a graded commutative algebra over $R$ isomorphic to the cohomology ring whose coefficient ring is $R$ of a bouquet of a finite number of closed and connected manifolds whose dimensions are smaller than $n$ and which can be embedded into ${\mathbb{R}}^n$.

By a connected sum of $f$ and a special generic map into ${\mathbb{R}}^n$ such that the restriction to the singular set is an embedding, we can construct a stable fold map $f_0$. By an ATSS operation to $f_0$, we have a new fold map $f^{\prime}:M^{\prime} \rightarrow N$ satisfying the following properties.
\begin{enumerate}
\item The cohomology ring $H^{\ast}(W_{f^{\prime}};R)$ is isomorphic to and identified with a graded commutative algebra obtained canonically from the direct sum of $H^{\ast}(W_{f};R)$ and a graded commutative algebra $B_R$ over $R$: we denote the $i$-th module of $B_R$ by $B_{R,i}$ and we abuse this identification in explaining the products and so on.
\item $A_R$ is regarded as a subalgebra of $B_R$. We denote the $i$-th module of $A_R$ by $A_{R,i}$ and as a module over $R$, $B_{R,i}$ is represented as a direct sum of $A_{R,i}$ and a suitable module $C_{R,i}$ over $R$: we identify $B_{R,i}$ and $A_{R,i} \oplus C_{R,i}$ via a suitable isomorphism between modules over $R$ in the remaining properties.
\item $C_{R,i}$ is zero for $i \neq k,n-k$, $R \oplus R$ for $i=k=\frac{n}{2}$ and $i=n$, $R$ for $i=k,n-k$ where $k \neq \frac{n}{2}$.
\item For $k=\frac{n}{2}$, the product of $(0,1,0) \in A_{R,k} \oplus C_{R,k}=A_{R,k} \oplus R \oplus R$ and $(0,0,1) \in A_{R,k} \oplus C_{R,k}=A_{R,k} \oplus R \oplus R$ is $(0,t,t) \in A_{R,n} \oplus C_{R,n}=A_{R,n} \oplus R \oplus R$ and the squares of $(0,1,0) \in A_{R,k} \oplus C_{R,k}$ and $(0,0,1) \in A_{R,k} \oplus C_{R,k}$ vanish. For $k \neq \frac{n}{2}$, the product of $(0,1) \in A_{R,k} \oplus C_{R,k}=A_{R,k} \oplus R$ and $(0,1) \in A_{R,n-k} \oplus C_{R,n-k}= A_{R,n-k} \oplus R$ is $(0,t,t) \in A_{R,n} \oplus C_{R,n}=A_{R,n} \oplus R \oplus R$. Here the square of each of these two elements $(0,1)$ vanishes.
\item For $k=\frac{n}{2}$, the product of $(0,1,0) \in A_{R,k} \oplus C_{R,k}=A_{R,k} \oplus R \oplus R$ and any element $(a,0) \in A_{R,i} \oplus C_{R,i}$ for any $i$ vanishes. For $k \neq \frac{n}{2}$, the product of $(0,1) \in A_{R,k} \oplus C_{R,k}=A_{R,k} \oplus R$ and any element $(a,0) \in A_{R,i} \oplus C_{R,i}$ for any $i$ vanishes.
\end{enumerate}
Furthermore, in addition, let $A_{R,k}$ be not zero and let $a_k \in A_{R,k}$ be a non-zero element such that for any element $r \in R$ which is not a unit, we cannot represent as $a_k=r{a_k}^{\prime}$ for an element ${a_k}^{\prime} \in A_{R,k}$, that $ra_k=0$ if and only if $r \in R$ is zero, and that the homology class whose dual is $a_k$ is represented by a standard sphere of dimension $k$ embedded in the interior of the smooth manifold represented as a regular neighborhood of bouquet of a finite number of closed and connected manifolds in the beginning. Let $t^{\prime} \in R$ be represented as ${t_0}^{\prime} \in \mathbb{Z}$ times the identity element $1 \in R$.
We can construct the map satisfying either of the following property in addition to the five properties before.
\begin{enumerate}
\item For $k=\frac{n}{2}$, the product of $(a_k,0,0) \in A_{R,k} \oplus C_{R,k}=A_{R,k} \oplus R \oplus R$ and $(0,0,1) \in A_{R,k} \oplus C_{R,k}=A_{R,k} \oplus R \oplus R$ is $(0,0,t^{\prime}) \in A_{R,n} \oplus C_{R,n}= A_{R,n} \oplus R \oplus R$.
\item For $k \neq \frac{n}{2}$, the product of $(a_k,0) \in A_{R,k} \oplus C_{R,k}=A_{R,k} \oplus R$ and $(0,1) \in A_{R,n-k} \oplus C_{R,n-k}=A_{R,n-k} \oplus R$ is $(0,0,t^{\prime}) \in A_{R,n} \oplus C_{R,n}=A_{R,n} \oplus R \oplus R$.
\end{enumerate}

\end{Thm}
\begin{proof}
We have an $n$-dimensional compact manifold embedded into ${\mathbb{R}}^n$ whose cohomology ring is isomorphic to $A_R$ by the assumption. Example \ref{ex:1} (\ref{ex:1.1}) yields a special generic map into ${\mathbb{R}}^n$ whose Reeb space is diffeomorphic to the $n$-dimensional manifold. We can obtain the special generic map such that the restriction to the singular set is an embedding and that the two kinds of the presented bundles in Example \ref{ex:1} (\ref{ex:1.1}) are trivial.  

We consider a connected sum of $f$ and this special generic map. The resulting map is $f_0$. $H^{\ast}(W_{f_0};R)$ is a graded commutative algebra obtained canonically from the direct sum of $H^{\ast}(W_{f};R)$ and a graded algebra $A_R$ over $R$:

The key ingredient is the proof of Proposition \ref{prop:2}. This yields the first five properties. We explain the proof of the remaining property. $a_k$ is, by the definition and construction of special generic map $f_0$, regarded as the dual of a homology class represented by a standard $k$-dimensional sphere ${S^k}_0$ embedded in the regular value set of $f_0$ and $f_0 {\mid}_{{f_0}^{-1}({S^k}_0)}:{f_0}^{-1}({S^k}_0) \rightarrow {S^k}_0$ gives a trivial $S^{m-n}$-bundle. We can represent $l_0$ times the class $a_k$ by a standard $k$-dimensional sphere. In the proof of Proposition \ref{prop:2}, we take $\{(S_1,N(S_1),c_1:N(S_1) \rightarrow P), (S_2,N(S_2),c_2:N(S_2) \rightarrow P)\}$ so that $c_1(S_1)$ is the $k$-dimensional standard sphere. We can see that we can obtain the map satisfying the last two properties on products of the cohomology ring by observing the topologies of the resulting Reeb spaces: for more precise discussions, see the proofs of several propositions and theorems in \cite{kitazawa6}.
\end{proof}

\begin{Ex}
\label{ex:2}
The class of {\it CPS} manifolds are characterized as the minimal class of manifolds satisfying the following conditions {\rm (}see also \cite{kitazawa6}, in which the author introduced the class of these manifolds first{\rm )}.
\begin{enumerate}
\item A standard sphere whose dimension is positive is CPS.
\item A product of two CPS manifolds is CPS.
\item A manifold represented as a connected sum of two CPS manifolds is CPS.
\end{enumerate}
One of good properties of CPS manifolds is that we can embed these manifolds into Euclidean spaces as codimension $1$ closed submanifolds.
This is introduced as a reasonable class for obtaining special generic maps as in the beginning of the proof of Theorem \ref{thm:1} and Example \ref{ex:1} (\ref{ex:1.1}). This class also contributes to producing various explicit situations explaining Theorem \ref{thm:1} and several theorems previously obtained by the author well.
 
In \cite{kitazawa6}, in situations similar to that of Theorem \ref{thm:1}, ATSS operations are studied under the condition that graded commutative algebras playing roles as $A_R$ plays are isomorphic to some integral cohomology rings of bouquets of finite numbers of CPS manifolds and additional conditions that the immersions of the standard spheres are embeddings and that the images are disjoint: the numbers of embedded standard spheres are general.
Let $f$ be a canonical projection of a unit sphere of dimension $m>5$ into ${\mathbb{R}}^5$ and $A_R$ be an algebra over $R:= \mathbb{Z}$ isomorphic to the integral cohomology ring of $S^2 \times S^2$. Consider a case where $(S_1,S_2)=(S^2,S^3)$ and $(l_0,t_{0})$ is a general pair of integers. The fundamental group of the resulting Reeb space vanishes. In general, an integral cohomology ring isomorphic to that of this resulting Reeb space cannot be obtained as the integral cohomology ring of another resulting Reeb space if we consider cases such that the immersions of the standard spheres are embeddings and that the images are disjoint only in Theorem \ref{thm:1}. We see about this.

First consider a case where the bouquet of CPS manifolds is represented as $S^2 \times S^2$ satisfying the conditions that the immersions of standard spheres are embeddings, that the number of the immersions is two, that the two immersed spheres are $S^2$ and $S^3$, respectively, and that the images are disjoint. For the resulting Reeb space in this case, the 2nd integral cohomology group is isomorphic to $R \oplus R \oplus R=\mathbb{Z} \oplus \mathbb{Z} \oplus \mathbb{Z}$. The rank of the submodule consisting of all 2nd cohomology classes such that the products with 2nd cohomology classes are always zero is exactly $1$. For such a class, for any 3rd
 integral cohomology class, the product must vanish. Note that the integral homology group of the resulting Reeb space $W_{f^{\prime}}$ in this case is isomorphic to that of the Reeb space before, and that the fundamental group of the space vanishes. The rank of the image of the homomorphism defined by taking the product of an element of $H^2(W_{f^{\prime}};\mathbb{Q})$ and an element of $H^3(W_{f^{\prime}};\mathbb{Q})$ is at most $1$. In general, in the situation of Theorem \ref{thm:1} as before, the rank may be $2$.

Let the target space be ${\mathbb{R}}^5$. To obtain Reeb spaces whose integral homology groups are isomorphic to the integral homology groups of the Reeb spaces before, and whose fundamental groups vanish, such that the integral cohomology ring of $S^2 \times S^2$ is regarded as subalgebras of the integral cohomology rings of the Reeb spaces, $A_R$ must be a graded commutative algebra over $R:= \mathbb{Z}$ isomorphic to the integral cohomology ring of a bouquet of $S^2 \times S^2$ and $S^3$ or that of $S^2 \times S^2$ and $S^2$: we do not consider the previous case. In the former case, for the family of the embeddings of standard spheres, the number of the embeddings are two and the spheres are a point and a standard sphere of dimension $3$, respectively. In the latter case, for the family of embeddings of standard spheres, the number of the embeddings are two and the spheres are a point and a standard sphere of dimension $2$, respectively. For each case, the rank of the image of the homomorphism defined by taking the product of an element of $H^2(W_{f^{\prime}};\mathbb{Q})$ and an element of $H^3(W_{f^{\prime}};\mathbb{Q})$ is at most $1$. In general, in the situation of Theorem \ref{thm:1} as before, the rank may be $2$.
\end{Ex}
We have the following general facts respecting some of Example \ref{ex:2}.
\begin{Thm}
\label{thm:2}
In the situation of Theorem \ref{thm:1}, let $R:=\mathbb{Z}$ and let $2k<n$.
Assume that $A_{R,n-k}$ is zero. Assume also that the immersions of $S^k$ and $S^{n-k}$, respectively, are embeddings and that the images are disjoint. In this situation, consider about the graded commutative algebra $B_R$. The rank of the image of the homomorphism defined by taking the product of an element of $B_{R,k} \otimes \mathbb{Q}$ and an element of $B_{R,n-k} \otimes \mathbb{Q}$ is at most $1$.
\end{Thm}
\begin{proof}[The main ingredient of a proof]
In the situation of theorem \ref{thm:1}, $t$ must be $0$.
\end{proof}
\begin{Thm}
\label{thm:3}
In the situation of Theorem \ref{thm:1}, let $2k<n$.
Assume that $A_{R,n-k}$ is trivial and the submodule consisting of all elements of $A_k$ such that for any element of $A_k$ the product vanishes is also trivial. Assume also that the immersions of $S^k$ and $S^{n-k}$, respectively, are embeddings and that the images are disjoint. In this situation, consider about the graded commutative algebra $B_R$ and let ${B_{R,k}}^{\prime} \subset B_{R,k}$ be the submodule consisting of all elements such that for any element of $B_R$ the product vanishes. The rank of the image of the homomorphism defined by taking the product of an element of ${B_{R,k}}^{\prime}$ and an element of $B_{R,n-k}$ must be $0$. If we remove the conditions that the immersions of $S^k$ and $S^{n-k}$, respectively, are embeddings and that the images are disjoint, then the rank of the image of the homomorphism defined by taking the product of an element of ${B_{R,k}}^{\prime}$ and an element of $B_{R,n-k}$ is at most $1$.
\end{Thm}
\begin{proof}[The main ingredient of a proof]
For the former statement, in the situation of theorem \ref{thm:1}, $t$ must be $0$. The latter statement immediately follows from the situation.
\end{proof}
Last we introduce a proposition stating that the Reeb space of a stable fold map satisfying several properties on indices of singular points and preimages know much about homology groups and the cohomology ring of the manifold of the domain.
\begin{Prop}[\cite{saekisuzuoka}, \cite{kitazawa2}, \cite{kitazawa3}, \cite{kitazawa7} and so on.]
\label{prop:3}
For a stable fold map $f:M \rightarrow N$ from an $m$-dimensional closed, connected and orientable manifold $M$ into an $n$-dimensional manifold $N$ with no boundary such that the following properties hold.
\begin{enumerate}
\item $m-n>1$.
\item Preimages of regular values
 are always disjoint unions of standard spheres.
\item Indices of singular points are always $0$ or $1$.
\item For all the crossings of the restriction of $f$ to the singular set $S(f)$, which is an immersion, the preimages consist of exactly two points.
\end{enumerate}
In this situation, we have the following two properties for a commutative group $A$.
\begin{enumerate}
 \item Three induced homomorphisms ${q_f}_{\ast}:{\pi}_j(M) \rightarrow {\pi}_j(W_f)$, ${q_f}_{\ast}:H_j(M;A) \rightarrow H_j(W_f;A)$ and ${q_f}^{\ast}:H^j(W_f;A) \rightarrow H^j(M;A)$ are isomorphisms for $0 \leq j \leq m-n-1$. 
\item Suppose that $A$ is a commutative ring. Let $J$ be the set of all integers greater than or equal to $0$ and smaller than or equal to $m-n-1$ and if ${\oplus}_{j \in J} H^{j}(W_f;A)$ and ${\oplus}_{j \in J} H^{j}(M;A)$ are algebras such that the sums and the products are canonically induced from the cohomology rings $H^{\ast}(W_f;A)$ and $H^{\ast}(M;A)$ respectively and that the maximal degrees are $m-n-1$ (the product is zero if it is of degree larger than $m-n-1$), then $q_f$ induces an isomorphism between the algebras
${\oplus}_{j \in J} H^j(W_f;A)$ and ${\oplus}_{j \in J} H^{j}(M;A)$ and the isomorphism is given by the restriction of ${q_f}^{\ast}$ to ${\oplus}_{j \in J} H^j(W_f;A)$.
\item Suppose that $A$ is a principal ideal domain and that $m=2n$ holds. Under these assumptions, the rank of $H_n(M;A)$ is twice the rank of $H_n(W_f;A)$ and in addition if $H_{n-1}(W_f;A)$ is free, then the $H_{n-1}(M;R)$, $H_n(M;A)$ and $H_n(W_f;A)$ are also free modules over $A$.
\end{enumerate} 
\end{Prop} 

Hereafter, the description is essentially same as that of the last part of \cite{kitazawa7}.

The {\it piecewise smooth category} is the category such that objects are smooth manifolds having canonically defined PL structures and that morphisms are piecewise smooth maps between these manifolds with these PL structures. This category is known to be equivalent to the PL category.

We give a sketch of the proof of this. We can give more rigorous proofs this referring to the discussions in the first three referred articles and \cite{saeki3}.
\begin{proof}[A sketch of the proof]

For each point in the image and an $n$-dimensional small standard closed disc containing the point as its neighborhood, each connected component of the preimage is either of the following types as smooth manifolds which may have corners,
\begin{itemize}
\item A product of an ($m-n$)-dimensional standard sphere and an $n$-dimensional standard closed disc. 
\item A product of an manifold obtained by removing the interior of the disjoint union of three disjointly and smoothly embedded standard closed discs in $S^{m-n+1}$ and an ($n-1$)-dimensional standard closed disc.   
\item A product of an manifold obtained by removing the interior of the disjoint union of four disjointly and smoothly embedded standard closed discs in $S^{m-n+1}$, a closed interval $I$ and an ($n-2$)-dimensional standard closed disc. 
\end{itemize}

There exist three types of the topologies of small regular neighborhoods of points in the Reeb space. Each type of the topologies corresponds to each case above. All points in the Reeb space of these three types form manifolds whose dimensions are $n$, $n-1$, and $n-2$, respectively.

For each case, we can construct bundles whose fibers are as above in the piecewise linear category and we can construct bundles whose fibers are $D^{m-n+1}$, $D^{m-n+2}$, or $D^{m-n+2} \times I$ in the category whose subbundles obtained by restricting the fibers to suitable compact submanifolds of the boundaries of the discs $D^{m-n+1}$ or $D^{m-n+2}$ are the original bundles. Note that the dimensions of the suitable compact submanifolds are same as those of the boundaries and that in the last case first we restrict $D^{m-n+2} \times I$ to $\partial D^{m-n+2} \times I =S^{m-n+1} \times I$. We can locally construct these bundles and glue them in the piecewise smooth category or PL category.
As a result, we have a desired ($m+1$)-dimensional compact PL manifold collapsing to $W_f$. $W_f$ is an $n$-dimensional polyhedron. In the PL category, the resulting ($m+1$)-dimensional manifold is a PL manifold obtained by attaching handles whose indices are larger than or equal to $m-n$ to $M \times \{0\} \subset M \times [0,1]$. This essentially completes the proof.

\end{proof}
We can apply this to explicit fold maps obtained in Theorem \ref{thm:1} (Proposition \ref{prop:2})
 in suitable situations.

\section{Acknowledgement.}
\thanks{The author has been supported by JSPS KAKENHI Grant Number JP17H06128 "Innovative research of geometric topology and singularities of differentiable mappings"
 as a member of the project. The present work is also supproted by the project.}

\end{document}